\documentclass{amsart}

\usepackage{amsfonts}
\usepackage{latexsym}
\usepackage{amscd}
\usepackage{amsmath}
\usepackage{amssymb}
\usepackage[mathscr]{eucal}

\newtheorem{theorem}{Theorem}[section]
\newtheorem{lemma}[theorem]{Lemma}
\newtheorem{corollary}[theorem]{Corollary}

\theoremstyle{definition}
\newtheorem{definition}[theorem]{Definition}

\newtheorem{question}[theorem]{Question}

\theoremstyle{remark}

\numberwithin{equation}{section}

\def\F{\mathcal{F}}

\def\I{\mathcal{I}}

\def\G{\mathcal{G}}

\def\C{\mathcal{C}}
\def\w{\omega}
\def\A{\mathcal{A}}
\def\B{\mathcal{B}}
\def\I{\mathcal{I}}
\def\U{\mathcal{U}}
\def\V{\mathcal{V}}

\def\vac{\emptyset}
\def\c{\mathfrak{c}}
\def\P{\mathcal{P}}
\def\Q{\mathcal{Q}}

\begin{document}

\title[Ordering Fr\'echet filters]{Comparing Fr\'echet-Urysohn filters with two pre-orders}

\author{S. Garcia-Ferreira}
\address{Centro de Ciencias Matem\'aticas,   Universidad Nacional
Aut\'onoma de M\'exico, Campus Morelia, Apartado Postal 61-3, Santa Maria,
58089, Morelia, Michoac\'an, M\'exico}
\email{sgarcia@matmor.unam.mx}
\author{J. E. Rivera-G\'omez}
\address{Centro de Ciencias Matem\'aticas,   Universidad Nacional
Aut\'onoma de M\'exico, Campus Morelia, Apartado Postal 61-3, Santa Maria,
58089, Morelia, Michoac\'an, M\'exico}
\email{jonathan@matmor.unam.mx}

\thanks{Research of the first-named  author was supported
by  CONACYT grant no. 176202 and PAPIIT grant no. IN-101911}


\date{}

\dedicatory{}


\begin{abstract}  A filter $\F$ on $\w$ is called Fr\'echet-Urysohn   if the space with only one non-isolated point $\w \cup \{\F\}$ is a Fr\'echet-Urysohn space, where the neighborhoods of the non-isolated point are determined by the elements of $\F$. In this paper, we distinguish
some Fr\'echet-Urysohn filters by using two pre-orderings of filters: One is  the Rudin-Keisler pre-order and the other one was introduced by Todor\v{c}evi\'c-Uzc\'ategui in \cite{tu05}.
In this paper, we construct  an $RK$-chain of size  $\c^+$ which is $RK$-above of avery $FU$-filter. Also, we show that there is an infinite $RK$-antichain of $FU$-filters.
\end{abstract}

\maketitle

\section{Notation and preliminaries}

All filters will be taken on $\w$ and  be free.  For an
infinite set $X$, we let $[X]^{< \w}=\{A\subseteq X:|A|< \w\}$ and  $[X]^{\w}=\{A\subseteq X:|A|=\w\}$. For $A, B\in [\w]^{\w}$,
$A\subseteq^*B$ means that $A \setminus B$ is finite.  If
 $S \in [\w]^\omega$,  we say that $S \to \F$ if $S \subseteq^* F$ for every $F \in \F$.
 If $\F$ is a filter, then $C(\F) = \{ S \in [\w]^\w : S \to \F \}$ is the set of all sequences converging to $\F$.
For a filter $\F$, we let $\I_\F = \{\w
\setminus F:F\in \F\}$ ({\it the dual ideal}) and for an ideal $\I$,
we let $\F_\I = \{\w \setminus I : I\in \I\}$ ({\it the  dual
filter}).  If $\F$ is a filter and
$f:\w \to \w$ is a function, then we  define the filter $f[\F]=\{ F  :f^{-1}(F)\in \F\}$.
 For $A \in [\w]^\w$, we define $\F_r(A) = \{ B \subseteq A : |A
\setminus B| < \w\}$. In particular,
$\F_r(\w) := \F_r$ is the  {\it Fr\'echet filter}.
 We say that an infinite  family $\A \subseteq [\w]^{\w}$ is {\it almost disjoint} ($AD$-family) if $A\cap B$ is finite for distinct $A, B \in \A$.
The ideal generated by an $AD$-family $\A$ is
$\I(\A)=\{ X \subseteq \w: \exists \A'\in[\A]^{< \w}(X \subseteq^*\bigcup \A')\}$.
An $AD$-family $\A$ is called
{\it maximal almost disjoint}  ($MAD$-{\it family}) if it is not contained
properly in another $AD$-family.   More general, if $\mathcal{B}
\subseteq [\w]^\w$, then we say that a family $\mathcal{A}$ is {\it
maximal} in $\mathcal{B}$ if $\mathcal{A} \subseteq \mathcal{B}$ and
for every $B \in \mathcal{B}$ there is $A \in \mathcal{A}$ such that
$|A \cap B| = \w$. For a nonempty $\A\subseteq [\w]^{\w}$, we define
$\A^{\bot}=\{B \in [\w]^{\w} : \forall A\in \A(|A\cap B|<\w)\}$,
$\A^+ = \{B \in[\w]^{\w}:\forall A\in \A(|A\cap B| \neq
\emptyset)\}$ and $\A^* = \{B \in[\w]^{\w}: |\{ A \in \A : |A \cap
B| = \w \}| \geq \w \}$. Notice that, for a filter $\F$, we have that  $\F^+ = \P(\w) \setminus \I_\F$ and,
 for an arbitrary ideal $\I$,  $\I^\perp$ is always an ideal and $\I
\subseteq \I^{\perp \perp}$.
For each $X \in[\w]^{\w}$ and  $\A \subseteq [\w]^{\w}$ we let
$\A|_{X}=\{ A \cap X : A \in \A \ \text{and} \ |A \cap X|=\w\}$.
Observe that if $\A$ is an $AD$-family and
$B \in \A^*$, then $\A|_B = \{ A \cap B : A \in \A \ \text{and} \ |A
\cap B| = \omega \}$ is an $AD$-family on $B$.

\medskip

 Now, let $\F$ be a filter on $\w$ and consider the space $\xi(\F) = \w\cup \{F\}$ whose topology is defined  as follows:
All elements of $\w$ are isolated and the neighborhoods of $\F$ are  of the form $\{\F\} \cup F$ where $F \in \F$.
One class of spaces which has been extensively studied in Topology is the following:

\smallskip

A space $X$ is called a {\it Fr\'echet-Urysohn} space (for short $FU$-space) if for each $x\in X$ such that $x\in cl_X A$, there is a sequence in $A$ converging to $x$.

\smallskip

\begin{definition} A filter $\F$ is called a $FU$-{\it filter}  if the space $\xi(\F)$ is Fr\'echet-Urysohn.
\end{definition}

The ``smallest'' $FU$-filter is the Fr\'echet filter $\F_r$ and the countable $FAN$-filter is also an example of a $FU$-filter which does not have a countable base. By using $AD$-families, in the paper \cite{gu09}, the authors pointed out the existence of  $2^\frak{c}$  pairwise non-equivalent $FU$-filters. In other terms, we have  that $\F$ is a $FU$-filter iff for every $A \in \F^+$ there is $S \in [A]^\w$ such that $S \subseteq^* F$ for all $F \in \F$.

\medskip

Two notions that will help us to distinguish $FU$-filters are the following.

\begin{definition}\label{tu} Let $\F$ and $\G$ be two filters on $\w$.
\begin{enumerate}
\item $\F\leq_{RK}\G$ if there is a function $f: \w \to \w$ such that $f[\G] = \F$ (i. e.,
$F\in \F$ iff $f^{-1}(F)\in \G$).

\item ({\bf \cite{tu05} })  $\F\leq_{TU}\G$ if there are $A\in \G^+$, $B\in \F$ and
 a bijection $f:A\to B$ such that $f[\G|_{B}]=\F|_A$.
\end{enumerate}
\end{definition}

We assert that these two relations $\leq_{RK}$ and $\leq_{TU}$ are reflexive and transitive but they are
not anti\-sy\-mmetric.

\begin{definition}
Let $\F$ and $\G$ be filters on $\w$.
\begin{enumerate}
\item  $\F\approx\G$ if there is a bijection function $f: \w \to \w$ such that
$f[\F]= \G$.

\item $\F\approx_{RK}\G$ if  $\F\leq_{RK}\G$ and $\G\leq_{RK}\F$.
\end{enumerate}
\end{definition}

The definition introduce in $(1)$ can be generalized as:
If $A,B\in[\w]^{\w}$ and $\F, \G$ are filters on $A$ and $B$,
respectively, then $\F$ and $\G$ are called {\it equivalent}
if there is a bijection $f: A \to B$ such that $f[\F] =
\G$. It is evident that $\F \approx\G$ implies that $\F\approx_{RK}\G$ for every pair of filters $\F$ and $\G$.
However, we do not know if the inverse implication holds for the class of $FU$-filters:

\begin{question}
Are there two $FU$-filters $\F$ and $\G$ such that $\F\approx_{RK}\G$ and  $\F \not\approx\G$ ?
\end{question}

By the symbol $\F <_{RK}\G$ we shall understand that $\F\leq_{RK}\G$ and  $\F \not\approx\G$.

\medskip

This paper is a continuation of the work done in the article \cite{gr01}.
The second section is devoted to recall some basic properties of the $FU$-filters. We give combinatorial properties which are equivalent to the
$RK$-order and  the $TU$-order. These equivalences will allow to construct $FU$-filters with some interesting properties. In the third section,
we show that if $\F\leq \F_{\P}$, then either $\F$ is  relatively equivalent to the Fr\'echet
filter or equivalent to $\F_{\P}$. The pre-orders $\leq_{RK}$ and $\leq_{TU}$ are compared in the forth section. We show that $\leq_{TU} \subseteq \leq_{RK}$ in the category of $FU$-filters.
We also prove that
if $\A$ is a $NMAD$-family of size $\c$ which is completely separable, then
$S_{\P}\nleq_{TU}S_{\A}$ and $S_{\P} \leq_{RK} S_{\A}$. In the fifth section,  we construct  an $RK$-chain of size  $\c^+$ which is $RK$-above of each $FU$-filter.
The sixth section is devoted to study the $RK$-incomparability of $FU$-filters. We use the $\alpha_i$ properties to show the $RK$- incomparability of  certain $FU$-filters. Besides,  we construct an infinite $RK$-antichain of $FU$-filters.

\section{Fr\'echet-Urysohn Filters}

In order to study the $FU$-filters and their relationships we list some useful facts.

\medskip

Let $\F$ be a filter on $\w$.
\begin{enumerate}
\item If $A\in[\w]^{\w}$, then  $\F\in cl_{\xi(\F)}A$ iff $A\in \F^+$.
\item $S \to \F$ iff $S \in \I_\F^{\perp}$.
\item $\F$ is a $FU$-filter iff $\I_\F^{\perp \perp} = \I_\F$.
\end{enumerate}
 It is not hard to prove that if $\emptyset \neq \mathcal{D}\subseteq [\w]^{\w}$, then
$$
\F_{\mathcal{D}}=\{F\subseteq \w:\forall D\in \mathcal{D}(D\subseteq^*F)\}
$$
is a $FU$-filter. By using this kind of filters and $AD$-families it is possible to characterize the  $FU$-filters as follows:

\begin{lemma}\label{simon}{\bf \cite{si98}} A filter $\F$ is a $FU$-filter iff there is an $AD$-family $\A$ maximal in $I^\perp_{\F}$ such that $\F=\F_{\A}$.
\end{lemma}

We can see directly from this characterization that  $\A$ is a $MAD$-family iff $\F_{\A} = \F_r$. More general,
if  $\F_{\A}$ is the filter generated by an $AD$-family $\A$ and $B\to \F_{\A}$, then there is $A\in \A$ such that $|A\cap B|=\w$.
Observe  from Lemma \ref{simon} that for every infinite $\mathcal{D}\subseteq [\w]^{\w}$, we can find an $AD$-family $\A$ such that
$\F_{\mathcal{D}} = \F_{\mathcal{A}}$. In what follows, when we write $\F_{\mathcal{A}}$ we shall always assume that $\A$ is an  $AD$-family.

\medskip

Given a $FU$-filter $\G$, we say that $\G$ is {\it relatively equivalent} to the Fr\'echet
filter iff  $A \to \G$ and $A \in \G$. In particular, we have that
$\G = \{ G \cup E : G \in \F_r(A) \ \text{and} \ E \subseteq \w \setminus A \}$. If
 we do not require that the function $f$ to be onto in the definition of the $RK$-order, then we would have that
 $\F_r(A) \leq_{RK} \F_r$ for every $A \in [\w]^{\w}$. For our convenience, in the definition of $RK$-order, we shall always require
that the function involved  be onto. This convenience  is based on  the next theorem which was proved in \cite{gr01}.

\begin{theorem}\label{new-RK} Let $\F$ and $\G$ filters such that
$\G \neq \F_r$ and $\F$ is not  relatively equivalent to the Fr\'echet
filter. If $\F \leq_{RK} \G$, then there is a surjective function $g: \w \to \w$ such that $g[\G] = \F$.
\end{theorem}

In virtue of the previous theorem, we remark that if $\F\leq_{RK}\F_r$, then $\F=\F_r$.
It was pointed out in \cite{gr01} that every filter $\G$
which has a nontrivial convergent sequence satisfies that $\F_r\leq_{RK} \G$.

\medskip

Next, let us describe  another useful construction of $FU$-filters which has been very important in the construction of special $FU$-filters (see, for instance, the article \cite{si}):

\medskip

For every $AD$-family $\A$, we define  $S_{\A}=\F_{\I(\A)}$. In general, $S_{\A}$ is not a $FU$-filter, for instance if $\A$
is a $MAD$-family, then $S_{\A}$ does not have any nontrivial convergent sequence. In order that $S_{\A}$ be a $FU$-filter we need some special $AD$-families:

\medskip

 An $AD$-family $\A$  is said to be {\it nowhere $MAD$-family} ($NMAD$-family) if
for every $X \in \I(\A)^+$ there is $A \in \I(\A)^{\bot}\cap [X]^{\w}$.
We remark that if $\A$ is a $NMAD$-family, then   every infinite
subfamily $\B$  of $\A$ is also a $NMAD$-family.

\begin{theorem}\label{sfu} ({\bf \cite{si}}) Given an $AD$-family $\A$, we have that the filter $S_{\A}$ is a $FU$-filter iff $\A$ is a $NMAD$-family.
\end{theorem}

\medskip

If $\P=\{P_n:n<\w\}\subseteq [\w]^{\w}$ is an infinite partition of $\w$ in infinite subsets, by Theorem \ref{sfu}, then $\F_{\P}$ is an $FU$-filter which is known as the  $FAN$-filter. We also know that $S_{\P}$ and $\F_r$ are the only  $FU$-filters with countable base. In what follows, when we write $S_{\P}$ we shall understand that $\P$ is an infinite partition of $\w$  in infinite subsets. For the  filters of the form $S_{\A}$ is very easy to know their characters as we shall see next.

\begin{lemma}\label{char-nmad} If  $\A$ is an $AD$-family on $\w$, then $\chi(S_{\A})=|\A|$.
\end{lemma}

The following combinatorial statements are equivalent  to the $RK$-order and were proved in \cite[Th. 3.4]{gr01}.

\begin{theorem}\label{theoeq} Let $\A$ and $\B$ be $AD$-families on $\w$.
The following statements are equivalent.
\begin{enumerate}
\item $\F_{\A}\leq_{RK}\F_{\B}$ via the function $f: \w \to \w$.
\item
\begin{enumerate}
 \item $\forall F\in \F_{\A}(f^{-1}(F)\in \F_{\B})$, and
 \item $\forall G\in \F_{\B}(f[G]\in \F_{\A})$.
\end{enumerate}
\item
\begin{enumerate}
\item $\forall n<\w \forall B\in \B(|f^{-1}(n)\cap B|<\w)$,

\item $\forall B\in \B \forall C\in  \A^{\bot}(|f[B]\cap C|<\w)$, and

\item $\forall S \in C(\A) \exists B\in \B(|f^{-1}(S)\cap B|=\w)$.
\end{enumerate}
\end{enumerate}
\end{theorem}

The next result can be obtained by a slight modification of the proof of the previous theorem.

\begin{theorem}\label{theoeq2} Let $\A$ and $\B$ be $NMAD$-families on $\w$.
The following statements are equivalent.
\begin{enumerate}
 \item $S_{\A}\leq_{RK}S_{\B}$ via the function $f: \w \to \w$.
\item
\begin{enumerate}
 \item $\forall F\in S_{\A}(f^{-1}(F)\in S_{\B})$, and
 \item $\forall G\in S_{\B}(f[G]\in S_{\A})$.
\end{enumerate}
\item
\begin{enumerate}
\item $\forall I\in \I(\A) (f^{-1}(I)\in \I(\B))$, and
\item $\forall M \in \I(\A)^{\bot}\cap[\w]^{\w} \exists R\in \I(\B)^{\bot}\cap[\w]^{\w}(|f^{-1}(M)\cap R|=\w)$.
\end{enumerate}
\end{enumerate}
\end{theorem}

the last two theorems will be very useful in the construction of spacial filters.

\section{$FAN$-filter}

In what follows, $\P$ will stand for an infinite partition of $\w$ in infinite subsets. In the paper \cite{gr01}, we proved that the $FAN$-filter $\F_\P$ and the $S_\P$ filter are not $RK$-comparable. We know, by lemma \ref{char-nmad}, that the only filters which are $RK$-predecessors of
$S_{\P}$ are either itself or a filter that is relatively equivalent to the Fr\'echet filter. All these remarks
lead us to ask in \cite[Q. 5.12]{gr01}  what are the $RK$-predecessors of the $FAN$-filter ?:

\medskip

{\bf Question.} Is there an $AD$-family $\A$ such that $\F_r <_{RK} \F_\A <_{RK} \F_\P$ ?

\medskip

We shall respond this question in the next theorem.

\begin{lemma}
Let $\A$ be an $AD$-family on $\w$ and $f: \w \to \w$ be a surjective function such that
the restriction  $f|_{\A}$ is finite-to-one for each $A\in \A$.
If $\mathcal{D}_f :=\{f[A]:A\in\A\}$, then $\F_{\mathcal{D}_f}\leq_{RK}\F_A$ via $f$.
\end{lemma}

\begin{proof}
We have to prove that $\F_{\mathcal{D}_f}=f[\F_{\A}]$. Let $F\in\F_{\mathcal{D}_f}$ and assume
that $F\notin f[\F_{\A}]$. Then there is $A\in \A$ such that
$A\setminus f^{-1}(F)$ is infinite. Then, we have that $f[A\setminus f^{-1}(F)]$ is infinite
 and $f[A\setminus f^{-1}(F)]\to \F_{D_f}$, but $f[A\setminus f^{-1}(F)]\cap F=\vac$ which is
 impossible. Thus $F\in f[\F_{\A}]$. So, we obtain that $\F_{\mathcal{D}_f} \subseteq f[\F_{\A}]$.
Now fix $H\in f[\F_{\A}]$ and suppose that $H\notin \F_{\mathcal{D}_f}$. Then there is $A\in \A$ such that
$f[A]\setminus H$ is infinite. We know that $A\subseteq^*f^{-1}(H)$ which implies that $f[A]\subseteq^* H$, but this is a contradiction.
Thus $H\in \F_{\mathcal{D}_f}$. This proves that $f[\F_{\A}] \subseteq \F_{\mathcal{D}_f}$. Therefore,
$\F_{\mathcal{D}_f}=f[\F_{\A}]$ and hence $\F_{\mathcal{D}_f}\leq_{RK}\F_A$ via the function $f$.
 \end{proof}

\begin{lemma}
 If $\F_{\A}\leq_{RK}\F_{\B}$ via the function $f$, then $\F_{\mathcal{D}_f} = \F_{\A}$.
\end{lemma}

\begin{proof}
Observe that for each $B\in \B$, $f[B]\subseteq^* F$ for all $F\in \F_{\A}$. Hence, we have that  $\F_{\A}\subseteq \F_{\mathcal{D}_f}$.
Now let $G\in \F_{\mathcal{D}_f}$ and suppose that $G \notin \F_{\A}$. Then we can find $A\in \A$ such that
$|A\setminus G|=\w$. Since $A\setminus G\to \F_{\A}$, by Theorem \ref{theoeq}, there is $B\in \B$ such that $B\cap (f^{-1}(A)\setminus f^{-1}(G))$ is infinite,
 but this contradicts the fact $B\subseteq^* f^{-1}(G)$. Thus, we must have that $G\in \F_A$. Therefore, $\F_{\mathcal{D}_f} = \F_{\A}$.
\end{proof}

\begin{theorem}\label{theofan}
If $\F\leq_{RK} \F_{\P}$, then either $\F$ is  relatively equivalent to the Fr\'echet
filter or equivalent to $\F_{\P}$.
\end{theorem}

\begin{proof} It is known that $\F$ is an $FU$-filter (for a proof see \cite{gr01}).
Let $f:\w \to \w$ be a function such that $f[\F_{\P}]=\F$ and $\P=\{P_n:n<\w\}$.
Notice that $\mathcal{D}_f=\{f[P_n]:n<\w\}$ is a cover of $\w$ and $\F_{\mathcal{D}_f}=\F$ by the previous lemma.
First assume that there is $n < \w$ such that $|f[P_{m}] \setminus \big(\bigcup_{i<n} f[P_i] \big)| < \w$ for all $n < m < \omega$.
It is clearly that $\bigcup_{i<n_k} f[P_i]=A\in \F$ and since $A\to \F$, we also have that
$\F$ is  relatively equivalent to the Fr\'echet filter. Now,
set $Q_0=f[P_0]$ and inductively define
a pairwise disjoint family $\Q=\{ Q_k : k < \w\}$ of infinite subsets of $\w$, and a strictly increasing sequence $(n_k)_{k < \w}$ in $\omega$ so that $n_0 = 0$,
$Q_{k}=f[P_{n_{k}}]\setminus \big(\bigcup_{i<n_{k-1}} Q_i\big)$ is infinite and $n_{k}$ is the smallest with this property.
Since each element of $\Q$ is contained in an element of $\mathcal{D}_f$, we have that $\F_{\mathcal{D}_f}\subseteq \F_{\Q}$.
Let $F\in \F_{\Q}$ and fix $P_n\in \P$. Let
$k=min\{ i < \w : n \leq n_i <\w\}$. Then, by construction, we obtain that
$f[P_n]\subseteq^*\bigcup_{i\leq n_k} Q_i\subseteq^*F$. Hence, $F \in \F$. Therefore, $\F=\F_{\Q}$.
\end{proof}

To finish this section we pose the following question.

\begin{question} Let $\F$ be a $FU$-filter non-equivalent to the Fr\'echet
filter. If $\G \leq_{RK} \F$ implies that either $\G$ is  relatively equivalent to the Fr\'echet
filter or equivalent  to $\G$, must $\F$ be equivalent to the $FAN$-filter ?
\end{question}

\section{Todor\v{c}evi\'c-Uzc\'ategui pre-order}

In this section, we shall compare the $FU$-filters by using the  pre-order which has been introduced in Definition \ref{tu}. The first goal is to
compare this  pre-order with the $RK$-order. To do that we
reformulate the definition of the $TU$-order to make it a little bite easier to handle. Before this reformulation we need to prove a well-known property of certain  filters.

\begin{lemma}\label{lemref}
Let $\F$ be a filter non-relatively equivalent to the Fr\'echet filter.
Then for every $A\in \F\setminus \F_r$ there is a bijection  $f: A \to \w$ such that $f[\F|_{A}]=\F$. In particular,
$\F \approx \F|_A$ for all $A \in \F \setminus \F_r$.
\end{lemma}

\begin{proof} Let  $A\in\F\setminus \F_r$.
Since $\F$ is not relatively equivalent to the Fr\'echet filter there is
$B\in \F\cap [A]^{\w}$ such that $A\setminus B\in [A]^{\w}$. Define
$f:A\to \w$ so that $f|_B$ is the identity on $B$ and
$f[A\setminus B]=\w\setminus B$ as a bijection. Let $F\in \F$. Then we have that
$$f^{-1}(F)=f^{-1}(F\cap B)\cup f^{-1}(F\setminus B)=(F\cap B)\cup f^{-1}(F\setminus B).$$
Since $F\cap B\in \F|_A$, we must have that  $f^{-1}(F)\in \F|_A$. This shows that $\F\subseteq \F|_A$.
Now, fix $G\in \F|_A$. Notice that
$$f[G]=f[G\cap B]\cup f[G\setminus B]=(G\cap B) \cup f[G\setminus B].$$
As $G\in \F|_{A}$, there is $H\in \F$ such that $G=H\cap A$ and since
that $G\cap B= H\cap A \cap B\in \F$, then  $f[G]\in \F$. Thus, $\F|_A \subseteq \F$. Therefore, $f[\F|_{A}]=\F$.
\end{proof}

The next corollary conveniently reformulates the definition of the $TU$-order.

\begin{corollary} Let $\F$ and $\G$ two filters. Then,
$\F\leq_{TU} \G$ iff either
\begin{enumerate}
\item $\F$ is a relatively equivalent to the Fr\'echet filter, or

\item there are $A\in \G^+$ and a bijection $f:A\to \w$  such that
$f[\G|_A]=\F$.
\end{enumerate}
\end{corollary}

In virtue of the previous corollary,  we shall  always assume that  the element of the filter witnessing being a $TU$-predecessor  is  $\w$.
However, the  positive element that witnesses being a $TU$-successor cannot be replace by an element of the $FU$-filter as
in the $TU$-order: that is, if we use only members of the filters in the $TU$-order we obtain  the $RK$-order as it is shown in the next corollary.

\begin{corollary}
Let $\F$ and $\G$ be two filters non-relatively equivalent to the Fr\'echet filter. Then $\F\leq_{RK}\G$ iff
for each $A\in \G\setminus \F_r$ and $B\in \F\setminus \F_r$ there is a surjection $f:A\to B$ such that
$f[\G|_{A}]=\F|_B$.
\end{corollary}

We will see in the next theorem that the $TU$-order implies the $RK$-order whenever the $TU$-predecessor lies in the category
of the $FU$-filters.

\medskip

Let us remark that if $A\in \G^+\setminus \G$, then $\w\setminus A\in \G^+\setminus \G$ and
$$
\G=\G|_{A}\oplus\G|_{\w\setminus A}:=\{F\cup E: F\in \G|_{A} \ {\text and} \ E\in \G|_{\w\setminus A} \}.
$$

\begin{theorem}
Let $\F$ and $\G$ two filters such that $C(\F)\neq \vac$. If
$\F\leq_{TU} \G$, then $\F\leq_{RK} \G$.
\end{theorem}

\begin{proof}
Suppose that $A\in\G^+$ and $f:A\to \w$ witnesses that $f[\G|_A]=\F$. Fix $M\in C(\F)$.
 Define $g:\w\to\w$ so that $g|_A=f$ and $g[\w\setminus A]=M$
 as a bijection. Let $F\in \F$. Since every element of $M$ has
exactly two pre-images, we have that
$g^{-1}(F)=f^{-1}(F)\bigcup g^{-1}(F\cap M)$.
Clearly $f^{-1}(F)\in \G|_A$ and since $M\subseteq^{*} F$, then
$g^{-1}(F\cap M)\cap (\w\setminus A) \in \F_r(\w\setminus A)$. By the  above remark,
 we obtain that $g^{-1}(F)\in \G$. So  $\F\subseteq g[\G]$.
Now fix $G\in \G$. Then we have that
$$g[G]=g[G\cap A]\cup g[G\setminus A]=
f[G\cap A]\cup g[G\setminus A].
$$
Since $f[G\cap A]\in \F$, $g[G]\in \F$. Thus, $g[\G]\subseteq \F$. This proves that $g[\G]=\F$.
Therefore, $\F\leq_{RK} \G$.
\end{proof}

\begin{corollary}\label{coroturk}
Let $\F$ and $\G$ two $FU$-filters. If $\F \leq_{TU}\G$, then
$\F\leq_{RK}\G$.
\end{corollary}

In a general context, one may ask  what about the implications  $\leq_{RK} \Rightarrow \leq_{TU}$ and
$\leq_{TU} \Rightarrow \leq_{RK}$ without any restrictions on the filters ?
We just  have seen that the $TU$-order implies the $RK$-order on the class
of $FU$-filters, which is also true for every pair of ultrafilters
$\U, \V$ (i. e., $\U\leq_{TU} \V \Rightarrow \U\leq_{RK} \V$). Let us  see, in the next examples,
that both implications could fail  in general:

\smallskip

The {\it Arens filter}
$\F_a$ is defined by an infinite partition $\P=\{P_n:n<\w\}$ of $\w$
and the Fr\'echet filter on each $P_n$ as follows:

$$\F_a:=\{F\subseteq \w:\{n<\w:P_n\cap P\in \F_r(P_n)\}\in \F_r\}.$$

\noindent The filter $\F_a$ is sequential but it is not an $FU$-filter.
Besides, we know that the  filter $\F_a|_A$ is a copy of $\F_a$  for every $A\in \F_a^+$.
Thus if $\F\leq_{TU} \F_a$, then $\F\approx \F_a$.
However, the Fr\'echet filter is a $RK$-predecessor of the Arens filter
via the function $f:\w\to\w$, defined by $f[P_n]=n$ for each $n<\w$.
This example shows that the implication $\F\leq_{RK}\G\Rightarrow \F\leq_{TU}\G$ does not hold in general.
Now, we describe an example to show that the implication $\F\leq_{TU}\G\Rightarrow \F\leq_{RK} \G$ could be false.
Choose $M, N\in [\w]^{\w}$ so that $M\cap N=\vac$ and $M\cup N=\w$.
We know  that $\F_a \leq_{TU} \F_{a}(M)\oplus \G =\{F\cup G: F\in \F_a(M) \ \text{and} \ G\in \G\}$, where $\F_{a}(M)$ is a copy of the Arens filter on $M$ and
$\G$ is a filter on $N$. Let $\G$ be an arbitrary $FU$-filter and suppose that $\F_a\leq_{RK}\F_{a}\oplus \G$ via the function $f$.
Since $\G$ is an $FU$-filter, there is $R\in[N]^{\w}$ such that $R\to \F_{a}(M)\oplus \G$ and then  $f[R]\to \F_a$,
but $\F_a$ does not have any nontrivial convergent sequence. This shows that the implication
$\F\leq_{TU}\G\Rightarrow \F\leq_{RK}\G$ could fail in general.

\medskip

Next, we will use  $FU$-filters of the form $S_{\A}$ to show that $\leq_{RK}\nsubseteq \leq_{TU}$.
Before that we need to prove a theorem.

\begin{theorem}\label{theomini}
Let $\A$ and $\B$ be $NMAD$-families such that $S_{\A}\leq_{RK}S_{\B}$ via a function $f:\w\to \w$ which satisfies that
$|f^{-1}(n)|=\w$ for all $n<\w$. If $\C$ is a $NMAD$ family such that $\B\subseteq \C$,
then $S_{\A}\leq_{RK}S_{\C}$ via the function $f$.
\end{theorem}

\begin{proof}  In order to show that $f[S_{\C}] = S_{\A}$, we use clause $(3)$ of Theorem $\ref{theoeq2}$. By using the statement $(3)(a)$, we know that  $f^{-1}(n)\in \I(\B)$ for all $n<\w$. Hence, for each $n < \w$, we have that  $|f^{-1}(n)\cap C|<\w$ for every $C\in \C\setminus \B$.  Let $M\in \I(\A)^{\bot}$. Notice that the containment  $f[S_{\B}]\subseteq S_\A$ implies that there  is  $N\in \I(\B)^{\bot}$ such that $|f^{-1}(M)\cap N|=\w$.Thus, we obtain that $f^{-1}(M)\in S_\C^{+}=\P(\w)\setminus \I(\C)$. There are two cases to be consider. The first one is when $f^{-1}(M)=M_0\cup M_1$ where $M_0\in \I(\C)$ and $M_1\in \I(\C)^{\bot}\cap [\w]^{\w}$. This case is clearly done since $M_1$ does the job. For the second one assume that  $f^{-1}(M)\in \C^{*}$. Since $\C$ is a $NMAD$-family, then there is $N\in \I(\C)^{\bot}\cap [\w]^{\w}$ such that $|f^{-1}(M)\cap N|=\w$. Thus, $f[S_{\C}] = S_{\A}$.
Therefore, we conclude that  $S_{\A}\leq_{RK}S_{\C}$ via the function $f$.
\end{proof}

Let $\A$ be an $AD$ family on $\w$.
For each $A\in\A$ choose $E_A\in [\w]^{< \w}$ and consider one of the sets either $A'=A\cup E_A$ or $A'=A\setminus E_A$.
It is not  difficult to show that $S_{\A}=S_{\A'}$ where  $\A'=\{A':A\in \A\}$.
 For our convenience, without lose of generality, we shall assume that each $AD$-family $\A$ always contains
a partition $\{A_{n}:n<\w\}\subseteq \A$ of $\w$ in infinite subsets. We show next that $S_{\P}$ is an $RK$-minimal filter in the realm of the filters
of the form $S_{\A}$ where $\A$ is a $NMAD$-family.

\begin{corollary}\label{com}
  $S_{\P}\leq_{RK}S_{\A}$ for every
$NMAD$-family $\A$.
\end{corollary}

\begin{proof} Let $\A$ be an $AD$-family on $\w$ such that  $\A'=\{A_n:n<\w\}\subseteq \A$ is
a partition of $\w$. Enumerate $\P$ as $\{P_n:n<\w\}$ and
 we define $f:\w\to \w$ so that
$f[A_n]=P_n$  and  $|f^{-1}(n)|$ is infinite for each $n<\w$. It is straightforward
to prove that $S_{\P}\leq_{RK}S_{\A'}$ via $f$. Therefore,  from  Theorem \ref{theomini} we deduce that
$S_{\P}\leq_{RK}S_{\A}$ via the function $f$.
\end{proof}

The answer to the  following question  will be very useful to understand the filter $S_{\A}$.

\begin{question} Is true that
  $S_{\P}\leq_{RK}S_{\A}$ for every
$AD$-family $\A$ ?
\end{question}

We recall that an $AD$-family  is said to be {\it completely separable} if
for every $M\in \B^*$, there is $B\in \B$ such that $B\subseteq M$.
In the paper \cite{si}, P. Simon showed, in $ZFC$, the existence of a completely separable $NMAD$-family of size $\c$.

\begin{corollary}
If $\A$ is a $NMAD$-family of size $\c$ which is completely separable, then
$S_{\P}\nleq_{TU}S_{\A}$ and $S_{\P} \leq_{RK} S_{\A}$.
\end{corollary}

\begin{proof}
Let $\A$ be a completely separable $NMAD$-family of size $\c$. In the article \cite{si}, the author showed that this family satisfies that
$|\{A\in \A:|M\cap A|=\w\}|=\c$ for all $M\in \A^*$.
If $M\in S_{\A}^+$ and $|\{A\in \A:|M\cap A|=\w\}|< \w$, then $S_{\A}|_M$ is
a relatively equivalent to the Fr\'echet filter. If $M\in \A^*$, then
$S_{\A}|_M$ has character equal to $\c$.
Thus, $S_{\A}$ has not a copy of $S_{\P}$.
Therefore, $S_{\P}\nleq_{TU}S_{\A}.$
On the other hand,  by Corollary \ref{com}, we have that $S_{\P} \leq_{RK} S_{\A}$.
\end{proof}

The behavior of the filters of the form  $S_{\A}$, where $\A$ is completely separable, under the $RK$-order is not well-know yet. For instance, we
do not the answer to the following question.

\begin{question} Are there two completely separable  $NMAD$-families $\A$ and $\B$ such that their filters $S_{\A}$ and $S_{\B}$
are $RK$-incomparable ?
\end{question}

In the last section, we will construct two $NMAD$-families of size $\c$ whose  respective filters are $RK$-incomparable.

\section{Chains of $FU$-filters in the $RK$-order and $TU$-order.}

First, we shall describe an
operation of filters that preserves the $FU$-property and produces
$RK$-successors:

\smallskip

Let $I$ be a set, $\F$  a (not necessarily free) filter on $I$ and
$\mathcal{A} = \{ A_i : i \in I \}$  an $AD$-family. For
each $i \in I$,   choose a free filter $\F_i$ on the set $A_i$.
Then we define
$$
\sum_{\F}\F_i := \{ F \subseteq \w :\{i\in I: F\cap A_i \in
\F_i\}\in \F\}
$$
and
$$
\prod_{i \in I}\F_i := \{ F \subseteq \w : \forall i \in I(F\cap A_i \in
\F_i)\}.
$$
Notice that if the filter on $I$ is the trivial filter $\{I\}$, then
$$
\prod_{i \in I}\F_i  := \sum_{\{I\}}\F_i.
$$
The filter $\prod_{i \in I}\F_i$  is referred as the product of the filters $\{ F_i : i \in I\}$. Several interesting properties of this operation of filters are contained in \cite{gr01}.
It is evident that $\sum_{\F}\F_i$ is always a free filter on $\w$ and
that $F \in (\prod_{i \in I}\F_i)^+$ iff there is $i \in I$ such
that $F \in \F_i^+$. Hence, we deduce that $\prod_{i \in I}\F_i$ is
a $FU$-filter iff $\F_i$ is a $FU$-filter for all $i \in I$. We remark that $\sum_{\F}\F_i $ is not, in general,
an $FU$-filter: for instance the Arens filter $\F_a$. In this context,
the $FAN$-filter is the filter $\prod_{n < \w}\F_r(P_n)$
where $\{ P_n :n < \w\}$ is a partition of $\w$ in infinite subsets.
The product of finitely many filters $\F_0,....., \F_n$ will be denote by $\F_0
\oplus \F_1 \oplus ..... \oplus \F_n$. We point out that if $\mathcal{A} = \{ A_i : i \in I \}$ is an $AD$-family and $\A_i$ is an $AD$-family on $A_i$, for
each $i \in I$, then
$$
\prod_{i \in I}\F_{\A_i} = \F_{\bigcup_{i \in I}\A_i}.
$$

\medskip

To construct $RK$-up-directed chains we need the following lemma from \cite[4.2]{gr01}.

\begin{lemma}\label{many-su} Let $A\in[\w]^{\w}$. Suppose that $\A = \{ A_i : i \in I \}\cup \{A\}$ is an $AD$-family
and $\A_i$ is an $AD$-family on $\w$, for each $i \in I$. If
$f_i: \w \to A_i$ is a bijection, for every $i \in I$ and, $\B$
is an $AD$-family on $A$,  then
$\F_{\A_j} \leq_{RK} \F_{\bigcup_{i \in I}f_i[\A_i]\cup \B}$ for all $j \in
I$.
\end{lemma}

We remark that  $\leq_{RK}$ can be replaced by $\leq_{TU}$ in the previous Lemma.

\begin{theorem}
If $\{\A_{\xi}:\xi<\c\}$ is a collection of $AD$-families, then there is
an $AD$-family $\C$ such that $\F_{\A_{\xi}}<_{RK}\F_{\C}$ for all $\xi<\c$.
\end{theorem}

\begin{proof} Fix $A\in[\w]^{\w}$ so that $\w\setminus A$ is infinite and let
$\{A_{\xi}:\xi<\c\}$ be an $AD$-family on $\w\setminus A$. For each $\xi<\c$ choose
a bijection $f_{\xi}:\w\to A_{\xi}$. By the previous lemma we obtain that
$\F_{A_{\xi}}\leq_{RK}\F_{\cup_{\xi<\c}f[A_{\xi}]\cup \B}$ for all $\xi<\c$ and for every $AD$-family $\B$ on $A$.
We know that there are $2^{\c}$
pairwise distinct $AD$-families on $A$, and since  every filter $\F_{\A_{\xi}}$ has at most
$\c$-many  $RK$-predecessors,   we can find an $AD$-family $\B$ such that $\F_{\B} \nleq_{RK} \F_{\A_{\xi}} $ for all $\xi<\c$.
Therefore,  $\F_{\A_{\xi}}<_{RK}\F_{\C}$ for all $\xi<\c$, where  $\C=\bigcup_{\xi<\c}f[A_{\xi}]\cup \B$.
\end{proof}

\begin{corollary}
There is a strictly increasing $RK$-chain of $FU$-filters of size  $\c^+$ $RK$-above  every $FU$-filter.
\end{corollary}

\section{$RK$-Incomparability of $FU$-filters}

In this section, we construct an  $RK$-antichain consisting of $FU$-filters. The authors of \cite{tu05} have proved the existence of a
$TU$-antichain of size  $\c^+$ consisting of $FU$-filters.

\medskip

The next notions introduced by A. V. Arhangel'skii  in \cite{ar} will help us to distinguish several $FU$-filters.

\begin{definition}
Let $X$ be an space and $x\in X$. A {\it sheaf} of $x$ is a family of
sequences $\{C_n:n<\w\}$ in $X$ converging to $x$.
We say that  $x$ is an $\alpha_i$-{\it point} (for each $i=1,2,3,4$) if for every  sheaf $\{C_n:n<\w\}$ of $x$
there is a sequence $B$ converging to $x$ such that:
\begin{enumerate}
\item[] ($\alpha_1$) $C_n\subseteq^*B$, for all $n<\w$.
\item[] ($\alpha_2$) $C_n\subseteq^*B$, for all $n<\w$.
\item[] ($\alpha_3$) $|C_n\cap B|=\w$, for infinitely many $n<\w$.
\item[] ($\alpha_4$) $C_n\cap B\neq\vac$, for infinitely many $n<\w$.
\end{enumerate}
The space $X$ is called $\alpha_i$-{\it space} if every point in $X$ is an $\alpha_i$-point. In particular,
a filter $\F$ is an $\alpha_i$-{\it filter} if its nonisolated point is an
$\alpha_i$ point in the space $\xi(\F)$, for every $i=1,2,3,4$.
\end{definition}

It is straightforward to prove the following implications:
$$
 first \ countability \Rightarrow \alpha_1 \Rightarrow \alpha_2
\Rightarrow \alpha_3\Rightarrow \alpha_4.
$$
The $FAN$-filter is a canonical example of a $FU$-filter which
is not an $\alpha_4$-filter; indeed, it is well-know that a space is not an
$\alpha_4$-space iff the space contains a copy of $FAN$-space (for a prove see \cite{sw}).
In the article \cite{si}, P. Simon  constructed a completely separable $NMAD$-family $\A$ of size $\c$
such that  $S_{\A}$ is an $\alpha_4$-filter which is not an
$\alpha_3$-filter. For this $AD$-family $\A$, it is easy to show that $\F_{\A}$ is also an $\alpha_4$-filter that is not
an $\alpha_3$-filter.

\medskip

In the following, we shall use a standard well-known technic to construct $FU$-filters by using the Cantor tree $2^{<\w} = \bigcup_{n < \w}2^n$:

\medskip

For each  $x\in 2^{\w}$ we define $A_x=\{x|n:n<\w\}\subseteq 2^{<\w}$. For every
infinite $X\subseteq 2^{\w}$ we have that  $\A_{X}=\{A_x:x\in X\}$ is an $NMAD$-family
on $2^{<\w}$. By identifying $2^{<\w}$ with $\w$, the family  $\A_{X}$ can be considered as a family of subsets of $\w$.
P. Nyikos (\cite{ny1}) proved that $S_{\A_X}$ is an $\alpha_3$-filter
for all infinite  $X\subseteq 2^{\w}$. He also showed that there is $Z\subseteq 2^{\w}$
for which  $S_{\A_Z}$ is an $\alpha_2$-filter, but in general
this assertion could fail; for instance, $S_{\A_{2^{\w}}}$ is not an $\alpha_2$-filter.
All examples  of $FU$-filters given above lie in $ZFC$. Nyikos  have proved in \cite{ny2},
under the assumption  $\w_1=\mathfrak{b}$, that there is an $FU$-space that is
$\alpha_2$-space but it fails to be $\alpha_1$-space. In the same paper, it was proved that
if $\w_1<\mathfrak{b}$, then there is an $FU$-space which is $\alpha_1$ but it is not
a first countable space. Years later,  A. Dow (\cite{dow}) proved that the implication  ``$\alpha_2\Rightarrow \alpha_1$'' holds
inside of the  Lavers Model and together with J. Stepr\={a}ns \cite{dowst} constructed
a model of $ZFC$ in which every $\alpha_1$-space is a first countable space.
 The existence of an $\alpha_2$-space
which is not an $\alpha_1$-space, and the existence of an $\alpha_1$-space
which is not an first countable space are still open problems in $ZFC$.

\medskip

Now let us prove that the properties  $\alpha_2$, $\alpha_3$ and $\alpha_4$  are preserved
by the $RK$-order down-directed.

\begin{theorem}\label{theoalphas}
Let $\F_{\A}$ and $\F_{\B}$ be two $FU$-filters. If $\F_{\B}$ is an $\alpha_i$-filter
and $\F_{\A}\leq_{RK}\F_{\B}$, then $\F_{\A}$ is also an $\alpha_i$-filter, for each $i=2,3,4$.
\end{theorem}

\begin{proof}  We  only  give a proof for the $\alpha_2$-property  since the procedure for $\alpha_3$ and $\alpha_4$ is exactly the same.
Let $\{C_n:n<\w\}\subset C(\F_{\A})$ be a sheaf of $\F_{\A}$ and $f:\w\to\w$
such that $f[\F_{\B}]=\F_{\A}$. By Theorem \ref{theoeq}, we can find
$B_n\in\B$ such that $|f^{-1}(C_n)\cap B_n|=\w$ for every $n<\w$.
Notice that $\{f^{-1}(C_n)\cap B_n:n<\w\}$ is a sheaf of $\F_{\B}$.
Since $\F_{\B}$ is an $\alpha_2$-filter, then there is a sequence $B$ converging
to $\F_{\B}$ such that $|B\cap (f^{-1}(C_n)\cap B_n)|=\w$ for all
$n<\w$. We remark that $f[B]\to \F_{\A}$. Fix  $n<\w$. Let us  prove that $|f[B]\cap C_n|=\w$.
Indeed, we have that
$$
f[B\cap (f^{-1}(C_n)\cap B_n)]\subseteq f[B]\cap (C_n\cap f[B_n])
\subseteq f[B]\cap C_n.
$$
Since $B\cap (f^{-1}(C_n)\cap B_n) \to \F_{\B}$, then $|f[B\cap (f^{-1}(C_n)\cap B_n)]|=\w$
and so $|f[B]\cap C_n|=\w$. Therefore, $\F_{\A}$ is an
$\alpha_2$-filter.
\end{proof}

For an arbitrary $NMAD$-family $\A$, we know that the filter $\F_{\A}$ cannot be an $\alpha_3$-filter. Thus
we obtain the following corollary.

\begin{corollary}\label{a}
$S_\P$ is not an $RK$-successor of $\F_{\A}$ for any $NMAD$-family $\A$.
\end{corollary}

The next corollary is consequence a from Corollary \ref{a} and Corollary 5.6 from
\cite{gr01}.

\begin{corollary}
$S_\P$ is $RK$-incomparable with every filter $\F_{\A}$ such that $|\A|<\mathfrak{b}$.
\end{corollary}

By using the Corollary \ref{coroturk}, Theorem \ref{theoalphas} and some facts quoted above we obtain the next result.

\begin{corollary}
$\F_{\P}\leq_{RK} \F$ iff $\F_{\P}\leq_{TU} \F$.
\end{corollary}

Now we show that some of the $FU$-filters already described above are $RK$-incomparable.

\begin{theorem}
Let $\A$ be a $NMAD$-family completely separable  of size
$\c$. Then there is a set $X\subseteq 2^{\w}$ such that $\F_{\P}$, $S_{\A}$ and $S_{\A_X}$ form an
$RK$-antichain.
\end{theorem}

\begin{proof}
Notice that there are $2^{\c}$ pairwise non-homeomorphic filters of the form $S_{\A_X}$ whit $|X|=\c$.
We can choose one of them satisfying $S_{\A_X}\nleq_{RK} S_{\A}$.
We know that $S_{\A_X}$ is an $\alpha_3$-filter, $S_{\A}$ is an $\alpha_4$-filter
which is not an $\alpha_3$-filter and the $FAN$-filter $\F_{\P}$ is not an $\alpha_4$-filter. Hence,
by Theorem \ref{theoalphas}, we obtain that
$$
\F_{\P}\nleq_{RK} S_{\B}\nleq_{RK}  S_{\A_X} \ \text{and} \  \F_{\P}\nleq_{RK} S_{\A_{X}}.
$$
According to Theorem \ref{theofan}, we have that $S_{\A_X}\nleq_{RK}\F_{\P}$ and
$S_{\A}\nleq_{RK}\F_{\P}$. Therefore, $\F_{\P}$, $S_{\A}$ and $S_{\A_{X}}$ are pairwise $RK$-incomparable.
\end{proof}

Our next task is the construction of an infinite $RK$-antichain consisting of $FU$-filters.
Such  filters will be the form $S_{\A_X}$ for suitable sets $X\subseteq 2^{\w}$. For our purposes
it is important to remark the next characterization of the convergent sequences in $S_{\A_X}$:

\medskip

{\bf Remark.} For $X\in 2^{\w}$ and $N\in[2^{<\w}]^{\w}$, the following statements are equivalents:
\begin{enumerate}
\item $N\to S_{\A_X}$.
\item $N\in \I(\A_X)^{\bot}$.
\item For all $K\in[N]^{\w}$ there is either:
      \begin{enumerate}
			\item  $x\in 2^{\w}\setminus X$ such that $|A_x \cap K| = \omega$ or
      \item  an infinite antichain  $M$  such that $|M \cap K| = \omega$.
   		\end{enumerate}
\end{enumerate}
Thus, we may consider only branches and antichains of $2^{<\w}$.
The following equivalence is a consequence of the Theorem \ref{theoeq2} and our last remark.

\begin{lemma}\label{lemmi}
Let $X_0,X_1\subseteq 2^{\w}$ and $f:2^{\w}\to 2^{\w}$ a surjection. Then, $f[S_{\A_{X_1}}]\neq S_{\A_{X_0}}$ iff
one of the following conditions is satisfied:
\begin{enumerate}
\item  There is $x\in 2^{\w}\setminus X_1$ such that $f[A_x] \nrightarrow S_{\A_{X_0}}$.

\item  There is  an infinite antichain  $M$ such that $f[M] \nrightarrow S_{\A_{X_0}}$.

\item  There is $y\in 2^{\w}\setminus X_0$ such that $f^{-1}(A_y) \in \I(\A_{X_1})$.

\item  There is  an infinite antichain  $M$ such that $f^{-1}(M) \in \I(\A_{X_1})$.
\end{enumerate}
\end{lemma}

We would like to point out that clauses (1) and (2) imply $S_{\A_{X_0}}\nsubseteq f[S_{\A_{X_1}}]$, and conditions (3) and (4) imply
$f[S_{\A_{X_1}}]\nsubseteq S_{\A_{X_0}}$. If there is an infinite antichain $M$ such that $|f[M]|<\w$, then we may avoid this kind of
functions, since $f$ cannot be a witness of the $RK$-comparability for any pair of $FU$-filters.
Thus, in what follows, we shall always assume that  $f|_M$ is finite-to-one  at every antichain $M$.

\medskip

 Let us show in the next lemma that we can always  extend the sets $X_0$ and $X_1$ in order to
have witnesses for the $RK$-incomparability of their respective $FU$-filters of the extensions.

\begin{lemma}\label{lemant}
Let $X_0$ and $X_1$ be nonempty subsets of $2^{\w}$ such that $|2^{\w} \setminus (X_0 \cup X_1)| \geq \w$ and $f:2^{<\w}\to 2^{<\w}$ a surjection such that  $f|_M$ is finite-to-one for every infinite antichain $M$. Then, there are $X'_0,X_1', Y_0, Y_1  \subseteq 2^{\w}$ such that $X_0\subsetneq X_0'$,  $0 < |X_0'\setminus X_0|<\w$,
$X_1\subsetneq X_1'$, $0 < |X_1'\setminus X_1|<\w$, $0 < |Y_0|, |Y_1|< \omega$, $X_0'\cap Y_0=\vac = X_1'\cap Y_1$ and at least one of the following conditions holds:
\begin{enumerate}
\item[(a)]  There is $y\in Y_1$ such that $f[A_y] \nrightarrow S_{\A_{X_0'}}$.

\item[(b)]  There is  an infinite antichain  $M$ such that $f[M] \nrightarrow S_{\A_{X_0'}}$.

\item[(c)]  There is $x\in Y_0$ such that $f^{-1}(A_x) \in \I(\A_{X_1'})$.

\item[(d)]  There is  an infinite antichain  $M$ such that $f^{-1}(M) \in \I(\A_{X_1'})$.
\end{enumerate}
Thus, by  Lemma \ref{lemmi}, we have that $f[S_{\A_{X_1'}}]\neq S_{\A_{X_0'}}$.
\end{lemma}
\begin{proof} We need to consider two cases:

Case I. Suppose that  $f[S_{\A_{X_1}}]\neq S_{\A_{X_0}}$. Notice that if the witnesses of the $RK$-incomparability is an antichain satisfying
either  (2) or (4) of Lemma \ref{lemmi}, then we can extend arbitrarily $X_0\subseteq X_0'$, $X_1\subseteq X_1'$ and find $Y_0$, $Y_1$ such that $X_0'\cap Y_0=\vac$ and $X_1'\cap Y_1=\vac$ easily. Hence, either (b) or (d) holds.
Now suppose that the witness is  a branch that satisfies (1). There is $y\in 2^{\w}\setminus X_1$ and $x\in X_0$ such that $|f[A_y]\cap A_x|=\w$. Define $Y_1=\{y\}$, $X_1'=X_1\cup W$ where $W\subseteq 2^{\w}\setminus Y_1$, and
$X_0'$, $Y_0$ arbitrarily such that $X_0'\cap Y_0=\vac$. Thus we have (a). Assume now that the witness is  a branch satisfying (3). There is $v\in 2^{\w}\setminus X_0$ such that $f^{-1}(A_v)\in \I(\A_{X_1})$. Define $Y_0=\{v\}$, $X_0'=X_0\cup W$ where $W\subseteq 2^{\w}\setminus Y_0$, and
$X_1'$, $Y_1$ arbitrarily such that $X_1'\cap Y_1=\vac$. In this case (c) is satisfied.
In each case we have one of the conditions.

\medskip

Case II. Suppose that $f[S_{\A_{X_1}}]=S_{\A_{X_0}}$. We shall prove  that $X_0$ and $X_1$ can be extend and find $Y_0$ and $Y_1$ so that
their extensions will satisfy either  (b) or (c).
\begin{enumerate}
\item[(i)] Assume that there are $y\in 2^{\w}\setminus X_0$ and a nonempty finite set $\B\subseteq \{A_v:v\in 2^{\w}\}$ such that  $f^{-1}(A_z)\subseteq^*\bigcup \B$.
           Notice that $X_1\setminus \{v\in 2^{\w}: A_v\in \B\}\neq \vac$. In this case, we set $Y_0=\{y\}$, $X_0'=X_0 \cup W$ where $W\subseteq 2^{\w}\setminus Y_0$,  $X_1'=X_1\cup \{v\in 2^{\w}: A_v\in \B\}$ and $Y_1$ a finite
					nonempty set such that $X_1'\cap Y_1=\vac$. Thus, we have (c).
\item[(ii)] Now suppose that there is an antichain $M$ and $x\in 2^{\w}\setminus X_0$ such that $|f[M]\cap A_x|=\w$.
           In this case, we put $X_0'=X_0\cup \{x\}$, $Y_0$ a finite nonempty set such that $X_0'\cap Y_0=\vac$.          We can extend $X_1$ arbitrarily and find $Y_1$   a finite
					nonempty set such that $X_1'\cap Y_1=\vac$. Hence, we have (b0).
\end{enumerate}
If (i) and (ii) fail, then $f^{-1}(A_z)\in \I(\A_{2^{\w}})^+$, for each $z\in 2^{\w}\setminus X_0$, and for every infinite antichain $M$ we have that $f[M]\in \I(\A_{2^\w})^{\bot}$.
Hence  $f[M]$ cannot meet  any  branch $A_y$ in an infinite set, for all  $y\in 2^{\w}$. Fix $z\in 2^{\w}\setminus X_0$. As
$f^{-1}(A_z)\in \I(\A_{2^{\w}})^+$, then $f^{-1}(A_z)$ contains an infinite antichain $K$.
Since $f[K]$ is infinite and $f[K]\subseteq f[f^{-1}(A_z)]=A_z$,
we get a contradiction to the negation of (ii). Thus, either (i) or (ii) is satisfied.
\end{proof}

We are ready to  construct an infinite $RK$-antichain with  $FU$-filters of character equal to $\c$.

\begin{theorem}
For every infinite cardinal $\kappa<\c$, there is a family $\{X_{\alpha}:\alpha<\kappa \}\subseteq[2^{\w}]^{\c}$
such that $\{S_{X_{\alpha}}:\alpha < \kappa\}$ is an $RK$-antichain.
\end{theorem}
\begin{proof}
Let $\{f_{\beta}:\beta<\c\}$ be an enumeration of all  surjections $f_{\beta}:2^{<\w}\to 2^{<\w}$ for which $f_{\beta}|_M$ is finite-to-one for every infinite antichain $M$.
By an inductive procedure,  for every $\beta < \c$  we shall construct, for every $\alpha<\kappa$,  sets  $X_{\beta}^{\alpha}\subseteq 2^{\w}$ and $Y_{\beta}^{\alpha}\subseteq 2^{\w}$ so that:
\begin{enumerate}
    \item $X^{\alpha}_{\beta}\cap Y^{\alpha}_{\beta}=\vac$ for every $\alpha<\kappa$.
    \item $X^{\alpha}_{\mu}\subseteq X^{\alpha}_{\nu}$ and $Y^{\alpha}_{\mu}\subseteq Y^{\alpha}_{\nu}$ if $\mu < \nu < \c$.
      \item For  distinct
			       $\gamma, \delta <\kappa$ one of the following conditions holds:

      \begin{enumerate}
\item  There is $x\in Y^{\gamma}_{\beta+1}$ such that $f_\beta[A_x] \nrightarrow S_{\A_{X^{\delta}_{\beta+1}}}$.

\item  There is  an infinite antichain  $M$ such that $f_\beta[M] \nrightarrow S_{\A_{X^{\delta}_{\beta+1}}}$.

\item  There is $y\in Y^{\delta}_{\beta+1}$ such that $f_\beta^{-1}(A_y) \in \I(\A_{X^{\gamma}_{\beta+1}})$.

\item  There is  an infinite antichain  $M$ such that $f_\beta^{-1}(M) \in \I(\A_{X^{\gamma}_{\beta+1}})$.
\end{enumerate}

\item $|X^{\alpha}_{\beta}|, |Y^{\alpha}_{\beta}|\leq \kappa \cdot |\beta|$ for every $\alpha<\kappa$.
\end{enumerate}
Choose arbitrary distinct elements $x_0, y_0\in 2^{\w}$ and define $X_{0}^{\alpha}=\{x_0\}$ and $Y_{0}^{\alpha}=\{y_0\}$, for every $\alpha<\kappa$.
Assume that for $\beta<\c$ the sets
 $X_{\theta}^{\alpha}$ and $Y_{\theta}^{\alpha}$ have been defined for all $\theta < \beta$ and $\alpha<\kappa$ so that all of them satisfy the conditions (1), (2), (3) and (4). If $\beta<\c$ is a limit ordinal, then we define  $X_{\beta}^{\alpha}=\bigcup_{\theta<\beta} X_{\theta}^{\alpha}$ and
$Y_{\beta}^{\alpha}=\bigcup_{\theta<\beta} Y_{\theta}^{\alpha}$ for each $\alpha < \kappa$.
Now, suppose that $\beta = \theta +1$. We shall define $X_{\beta+1}^{\alpha}$ and $Y_{\beta+1}^{\alpha}$.  Notice from (4) that
$$
|2^{\w}\setminus \big[\big(\bigcup_{\alpha<\kappa} X_{\theta}^{\alpha}\big) \cup \big(\bigcup_{\alpha<\kappa}
Y_{\theta}^{\alpha}\big) \big]|= \c.
$$
Fix $\alpha < \kappa$. According to Lemma \ref{lemant}, for every $\gamma <  \kappa$ we can find finite nonempty  sets $B_{\gamma}^{\alpha}$, $C_{\gamma}^{\alpha}$, $D_{\gamma}^{\alpha}$ and $E_{\gamma}^{\alpha}$
so that the following conditions holds:
\begin{enumerate}
\item[(i)] $B_{\gamma}^{\alpha} \cup C_{\gamma}^{\alpha} \cup D_{\gamma}^{\alpha} \cup E_{\gamma}^{\alpha}\subseteq 2^{\w}\setminus \big[\big(\bigcup_{\alpha<\kappa} X_{\theta}^{\alpha}\big) \bigcup \big(\bigcup_{\alpha<\kappa}
Y_{\theta}^{\alpha}\big) \big].$

\item[(ii)] One of the following conditions hold
\begin{enumerate}
\item  There is $x\in Y^{\alpha}_{\theta}\cup D^{\alpha}_{\gamma}$ such that $f_\beta[A_x] \nrightarrow S_{\A_{X^{\gamma}_{\theta}\cup C^{\alpha}_{\gamma}}}$.

\item  There is  an infinite antichain  $M$ such that $f_\beta[M] \nrightarrow S_{\A_{X^{\gamma}_{\theta}\cup C^{\alpha}_{\gamma}}}$.

\item  There is $y\in Y^{\gamma}_{\theta}\cup E^{\alpha}_{\gamma}$ such that $f_\beta^{-1}(A_y) \in \I(\A_{X^{\alpha}_{\theta}\cup B^{\alpha}_{\gamma}})$.

\item  There is  an infinite antichain  $M$ such that $f_\beta^{-1}(M) \in \I(\A_{X^{\alpha}_{\theta}\cup B^{\alpha}_{\gamma}})$.
\end{enumerate}

\item[(iii)] $\big[(\bigcup_{\gamma\in \kappa\setminus \{\alpha\}}B_{\gamma}^{\alpha})\cap
\big[\bigcup_{\gamma\in \kappa\setminus \{\alpha\}}D_{\gamma}^{\alpha}\big]=\vac$

\item[(iv)] $\big[(\bigcup_{\gamma\in \kappa\setminus \{\alpha\}}C_{\gamma}^{\alpha})\cap
\big[\bigcup_{\gamma\in \kappa\setminus \{\alpha\}}E_{\gamma}^{\alpha}\big]=\vac$
\end{enumerate}
Define
$$
X_{\beta}^{\alpha}=X_{\theta}^{\alpha}\cup (\bigcup_{\gamma\in \kappa\setminus \{\alpha\}}B_{\gamma}^{\alpha})
\cup (\bigcup_{\gamma\in\kappa\setminus \{\alpha\}} C_{\alpha}^{\gamma})
$$ and
$$
Y_{\beta}^{\alpha}=Y_{\theta}^{\alpha}\cup (\bigcup_{\gamma\in \kappa\setminus \{\alpha\}}D_{\gamma}^{\alpha})
\cup (\bigcup_{\gamma\in\kappa\setminus \{\alpha\}} E_{\alpha}^{\gamma}).
$$
Conditions (1), (2), (3) and (4) are clearly satisfied. For every $\alpha<\kappa$ define $X_{\alpha}=\bigcup_{\beta<\c}X_{\beta}^{\alpha}$. Thus, by the construction and  Lemma
\ref{lemant} we have that, for every $\alpha, \gamma <\kappa$,
$$f_{\beta}[S_{\A_{X_{\alpha}}}]\neq S_{\A_{X_{\gamma}}}.$$
  Therefore $\{S_{X_{\alpha}}:\alpha<\kappa\}$ is an infinite $RK$-antichain of $FU$-filters which
	have character equal to $\c$.

\end{proof}

We end the paper with the following question that the authors could not solve it.

\begin{question}
Is there an $RK$-antichain of $FU$-filters of size  $\c$ ?
\end{question}

\bibliographystyle{amsplain}


\end{document}